\documentclass[5p]{elsarticle}
\usepackage[usenames, dvipsnames]{color}
\usepackage{lineno,hyperref}
\usepackage{epstopdf}
\usepackage{adjustbox}
\usepackage{amssymb}
\usepackage{amsmath,amsthm,amssymb,amsfonts}
\usepackage{subfigure}
\usepackage[titletoc,title]{appendix}

\everymath{\displaystyle}
\journal{Automatica}
\newcommand{\beqnum}{\begin{equation}\begin{array}{lcl}}
\newcommand{\eeqnum}{\end{array}\end{equation}}
\newcommand{\beqnom}{\begin{eqnarray}}
\newcommand{\eeqnom}{\end{eqnarray}}
\newcommand{\beqnc}{\begin{center}\begin{eqnarray}}
\newcommand{\eeqnc}{\end{eqnarray}\end{center}}
\newcommand{\beqnlm}{\begin{equation}\vspace{-.5cm}\begin{array}{lll}}
\newcommand{\eeqnlm}{\end{array}\end{equation}}\vspace{-.5cm}
\newcommand{\beq}{\begin{eqnarray*}}
\newcommand{\eeq}{\end{eqnarray*}}
\newcommand{\bef}{\begin{figure}[tbh!]}
\newcommand{\enf}{\end{figure}}
\newcommand{\vep}{\varepsilon}

\newcommand{\R}{\mathbb{R}}
\newcommand{\lf}{\left\lfloor}
\newcommand{\rr}{\right\rceil}
\newtheorem{montheo}{\bf Theorem}
\newtheorem{rem}{\bf Remark}
\newtheorem{lemme}{\bf Lemma}

\newtheorem{Assumption}{\bf Assumption}
\newtheorem{defn}{\bf Definition}

\begin{document}

\begin{frontmatter}

\title{{{\color{black}A Lyapunov Approach to Barrier-Function Based Time-Varying Gains Higher Order Sliding Mode Controllers}}\tnoteref{mytitlenote}}
%\title{Elsevier \LaTeX\ template}
\tnotetext[mytitlenote]{ This research was partially supported by the iCODE Institute, research project of the IDEX Paris-Saclay, and by the Hadamard Mathematics LabEx (LMH) through the grant number ANR-11-LABX-0056-LMH in the ``Programme des Investissements d'Avenir''.}

\author[mymainaddress]{Salah Laghrouche
\corref{mycorrespondingauthor}}
\cortext[mycorrespondingauthor]{Corresponding author}
\ead{salah.laghrouche@utbm.fr}
\author[mysecondaddress]{Mohamed Harmouche} 
\ead{mohamed.harmouche@actility.com}
\author[mythirdaddress]{Yacine Chitour}
\ead{yacine.chitour@lss.supelec.fr}
\author[mymainaddress]{Hussein Obeid}
\ead{hussein.obeid@univ-fcomte.fr}
\author[myfourthaddress]{Leonid Fridman}
\ead{Lfridman@unam.mx}

\address[mymainaddress]{Femto-ST UMR CNRS, Univ. Bourgogne Franche-Comt\'{e}/UTBM, 90010, Belfort, France}
\address[mysecondaddress]{Actility, Paris, France}
\address[mythirdaddress]{ L2S - Universite Paris Saclay, CentraleSupelec, CNRS 91192 Gif-sur-Yvette, France}
\address[myfourthaddress]{Departement of Robotics and Control, Engineering Faculty, Universidad Nacional Aut\'{o}noma de M\'{e}xico (UNAM), D.F 04510. M\'{e}xico}

\begin{abstract}
In this paper, we present Lyapunov-based {\color{black}time varying} controllers for  {\color{black}fast} stabilization of a perturbed chain of integrators with bounded uncertainties. We refer to such controllers as {\color{black}time varying} higher order sliding mode controllers since they are designed for nonlinear Single-Input-Single-Output (SISO) systems with bounded uncertainties such that the uncertainty bounds are unknown.
%{\color{blue} OLD: Our main result states that, given any neighborhood $\varepsilon$ of the origin, we determine a controller insuring, for every uncertainty bounds, that every trajectory of the corresponding closed loop system enters $\varepsilon$ and eventually remains there. Furthermore, based on the homogeneity property, a new asymptotic accuracy, which depends on the size of $\varepsilon$, is presented.}
{\color{black} We provide a time varying control feedback law insuring
verifying the following: there exists a family $(D(t))_{t\geq 0}$ of time varying open sets decreasing to the origin as $t$ tends to infinity, such that, for any unknown uncertainty bounds and trajectory $z(\cdot)$ of the corresponding system, there exists a positive positve $t_z$ for which $z(t_z)\in D(t_z)$ and $z(t)\in D(t)$ for $t\geq t_z$. 
%enters convergence in finite time of all the trajectories to a time varying domain $D(t)$ shrinking to the origin and their maintenance there. Hence, since the function $\eta(t)$ tends to zero, this leads the asymptotic convergence of all the trajectories to zero.
} 
The effectiveness of these controllers is illustrated through simulations.
\end{abstract}

\begin{keyword}
\texttt {\color{black}Time varying} sliding mode\sep Lyapunov-based {\color{black}time varying} controllers \sep Finite time stabilization \sep Perturbed integrator chain \sep Unknown bounded uncertainties. 
\end{keyword}

\end{frontmatter}

%\linenumbers

\section{Introduction}

%Parametric uncertainty in nonlinear dynamic physical systems arises from varying operating conditions and external perturbations that affect the physical characteristics of such systems. The variation limits or the bounds of this uncertainty might be known or unknown. This needs to be considered during control design so that the controller counteracts the effect of variations and guarantees performance under different operating conditions. Sliding mode control (SMC) \cite{Utkin,Slotine1984} is a well-known method for control of nonlinear systems, renowned for its insensitivity to parametric uncertainty and external disturbance. This technique is based on applying discontinuous control on a system which ensures convergence of the output function (sliding variable) in finite time to a manifold of the state-space, called the sliding manifold \cite{Young_Utkin}. In practice, SMC suffers from \emph{chattering}; the phenomenon of finite-frequency, finite-amplitude oscillations in the output which appear because the high-frequency switching excites unmodeled dynamics of the closed loop system \cite{Utkin1999}. 
%

Sliding Mode Control is an efficient tool for matched uncertainties compensation. Higher Order Sliding Mode Controllers (HOSMCs) \cite{Levant2001,Levant2003,Levant2005,ding2016simple,CRUZZAVALA2017232,7445158} allow to reduce the dynamics of systems of order $n$ with relative degree $r$ till $(n-r)$. 
The homogeneity properties of HOSMCs ensure the $r$-th order of asymptotic precision with respect to the sampling step and parasitic dynamics \cite{Levant2001,Levant2003,Levant2005,5406047}. To implement HOSMCs, the knowledge of upper bounds for $r$-th derivatives is needed.

In majority of real systems with relative degree $r$, the upper bounds of $r$-th derivatives of the outputs exist but they are unknown \cite{Plestan2010,Bartolini_IMA2013, Mor2016ACTA,edwards2016adaptive,EDWARDS2016183,Oliveira2018,Hsu2018,negrete2016second,Shtessel,SHTESSEL2017229,OBEID2018540,doi:10.1002/rnc.4253,doi:10.1002/rnc.4105,2016.1184759}. So adaptive HOSMCs design should satisfy two contradictory requirements:
\begin{itemize}
\item ensure a finite-time exact convergence to the origin;
\item avoid overestimation of the control gain \cite{Plestan2010}.
\end{itemize}

Recently, adaptive sliding mode controllers have attracted the interest of many researchers dealing with these requirements \cite{Ferreira_CST2011, Bartolini_IMA2013, Mor2016ACTA}. 
%Adaptive gains have been used with success in the past for chattering suppression. 
For example, in  \cite{utkin2013adaptive,edwards2016adaptive,EDWARDS2016183}  
adaptive approaches which are based on the Utkin's concept of equivalent control have proposed to adapt first order, second order and HOSMC algorithms. %These approaches consist in increasing the gain to enforce the sliding mode to be reached. After that, the equivalent control can be used as an estimation of the perturbation. 
{\color{black} However, the realization of this concept requires the knowledge of the equivalent control signal recovered from the filtration of the discontinuous control signal. It is possible to have this information if and only if the upper bound of the perturbations derivative is known. In this case, it is more useful to use continuous HOSMC.  }
 Inspired by this concept, an adaptive strategy has proposed in \cite{Oliveira2018} for first order sliding mode control. The advantages of this strategy are its simplicity and the possibility to be implemented in the case of
non smooth perturbations. Furthermore, with the aim of making less restrictive assumptions on the class of perturbations as well as maitaining the sliding mode, a monitoring function-based adaptive approach is proposed for first order sliding mode control in \cite{Hsu2018}.   
%For example, Bartolini et al. \cite{Bartolini_unknown} have extended Utkin's concept of equivalent control for second order sliding mode control gain adaptation, to suppress residual oscillations due to digital controllers with time delay. Similarly, an equivalent control based adaptive controller is described in \cite{XU2004}, in which the equivalent control estimation is improved, using double low pass filters. A concise survey of these methods can be found in \cite{Utkin_chatteringanalysis}. }

{ \color{black} On the other hand, another common approach which used dynamic gain adaptation is introduced by Huang et al. \cite{Huang2008} for first order sliding mode control. In this adaptation, the control gain increases until the sliding mode is achieved, and afterwards the gain becomes constant and stabilizes at unnecessarily large value. Inspired by this adaptation, Moreno et al. \cite{negrete2016second,Mor2016ACTA} have proposed an adaptation of different HOSMC algorithms that has the same aforementioned problem.  Plestan et al. \cite{Plestan2010,Plestan2012} have overcome this problem by decreasing the gains once the sliding mode is achieved. This method establishes real-sliding mode (convergence to a neighborhood of the sliding surface). However, it does not guarantee that the sliding variable would remain inside the neighborhood after convergence. In the field of HOSMC, Shtessel et al. \cite{Shtessel,SHTESSEL2017229} used this method to adapt super-twisting and twisting algorithms with non-overestimation of the control gains. In both algorithms, the sliding variable and its first time derivative converge to some unknown neighborhoods of zero.}

{ \color{black} Unlike these previous approaches, a barrier function-based adaptive strategy is proposed to overcome the problem of unknow domain of convergence  \cite{OBEID2018540}. Indeed, this strategy has been applied to adapt a first order sliding mode and super-twisting algorithms and it ensures the convergence of the output variable and maintains it in a predefined neighborhood of zero independent of the upper bounds of the perturbation and its derivative, without overestimating the control gain \cite{OBEID2018540,OBEID2019STC,OBEID2018IJC}. Inspired by this strategy, an adaptive twisting algorithm ensuring the convergence of the sliding variable and its first time derivative to some predefined neighborhoods of zero is presented in  \cite{8460272}. As a generalization of the results presented in \cite{OBEID2019STC}, an adaptive integral sliding mode control is designed in \cite{8619334} for a perturbed chain of integrators. This algorithm provides the convergence of the sliding variable and its $r-1$ first derivatives to some unknown neighborhoods of zero. To deal with this drawback, a barrier function-based dual layer strategy which can ensure the convergence of the sliding variable and its $r-1$ first derivatives to zero is proposed in  \cite{OBEID2019DHOSMC}. However, the two drawbacks of this strategy are that, \textcolor{black}{ on one hand, it cannot be applied in the case of time-dependent control gain, and on the other hand, the convergence can be lost if the perturbation grows suddenly, this means, the states can leave the origin and jump to a unknown big value before re-convergence.}}
  \\
%It can be noted that all research works avoiding gain overestimation, discussed so far, have yielded real sliding mode. In fact, real sliding mode is the only possibility when the uncertainty bounds are unknown, as the gain dynamics cannot respond immediately to sudden changes in system parameters.\\ 

%To the best of our knowledge, no contemporary work on adaptive HOSMC has been published for orders greater than two.
In this paper, we present Lyapunov-based {\color{black}time varying} controllers for the finite time stabilization of a perturbed chain of integrators with bounded uncertainties. Through a minor extension of the definition (as explained in the next section), we refer to such controllers as {\color{black}time varying} HOSMCs. \textcolor{black}{ The proposed {\color{black}time varying} controller guarantees the asymptotic convergence to the origin 
%finite time convergence to an adjustable arbitrary neighborhood of origin whose size does not depend upon the upper bounds of the uncertainties or their derivatives, 
i.e., it establishes real HOSM.} \textcolor{black}{ Indeed, it ensures the convergence of the states and their maintenance in a decreasing domain which tends to zero.} The main features of such {\color{black}time varying} controller can be summarized as follows: %{ \color{black} Moreover, compared to the earlier work  \cite{OBEID2019DHOSMC}, on one hand, , and on the other hand,  }

\begin{itemize} 
\item This controller can be extended to arbitrary order. 
\textcolor{black}{ \item The sliding variable and its $(r-1)$ first derivatives  converge in finite time to  a family of time varying open sets $(D(t))_{t\geq 0}$ decreasing to the origin as $t$ tends to infinity.
 \item Once a trajectory enters some $D(t_*)$ at time $t_*$, it remains trapped in the $D(t)$'s, i.e., $z(t)\in D(t)$ for $t\geq t_*$.
  \item The decrease to the origin of the family $(D(t))_{t\geq 0}$ can be chosen arbitrarily.
  \item We provide an explicit bound on the control which is best possible, i.e., (essentially) linear with respect to the uncertainty.}
\end{itemize}
The paper is organized as follows: problem formulation and {\color{black}time varying} controllers are presented in Section 2, simulation results which show the effectiveness of the proposed controllers are presented and discussed in Section 3. Some concluding remarks are given in Section 4.

%%%%%%%%%%%%%%%%%%%%%%%%%%%%%%%%%%%%%%%%%%%%%%%%%%%%%%%%%%%%%%%%%%%%%%%%%%%%%%%%

%%%%%%%%%%%%%%%%%%%%%%%%%%%%%%%%%%%%%%%%%%%%%%%%%%%%%%%%%%%%%%%%%%%%%%%%%%%%%%%%

\section{Higher Order Sliding Mode Controllers}

If $r$ is a positive integer, the perturbed chain of integrators of length $r$ corresponds to the (uncertain) control system given by 
\beqnum\label{u.l.s.}
\textcolor{black}{\dot z_i=z_{i+1},~ i=1,...,{r-1},}\quad \dot z_r =\varphi(t) + \gamma(t)u, %&\in&  I_\varphi+u I_\gamma.
\eeqnum
where $z=[z_1\  z_2\ ... z_r ]^T\in\mathbb{R}^r$, $u\in\mathbb{R}$ and the functions $\varphi$ and $\gamma$ are any measurable functions defined almost everywhere (a.e. for short) on $\mathbb{R}_+$ and bounded by positive constants $\bar \varphi$, $\gamma_m$ and $\gamma_M$, such that, for a.e. $t\geq 0$, 
\beqnum\label{bound}
\left| \varphi(t) \right| \le \bar \varphi,\quad
0 < \gamma_m\le \gamma(t) \le \gamma_M.
\eeqnum
One can equivalently define a perturbed chain of integrators of length $r$ as the differential inclusion $\dot z_r\in  I_{\bar\varphi}+u I_\gamma$ where $I_{\bar\varphi}=\left[-\bar\varphi , \bar\varphi \right]$ and $ I_\gamma=\left[\gamma_m , \gamma_M \right]$.

The usual objective regarding System \eqref{u.l.s.} consists of stabilizing it with respect to the origin in finite time, i.e., determining feedback laws $u=U(z)$ so that the trajectories of the corresponding closed-loop system converge to the origin in finite time. Note that, in general, the controllers $U(\cdot)$ are discontinuous and then, solutions of  \eqref{u.l.s.} need to be understood here in Filippov's sense \cite{Filippov}, i.e., the right-hand vector set is enlarged at the discontinuity points of the differential inclusion to the convex hull of the set of velocity vectors obtained by approaching $z$ from all directions in $\mathbb{R}^r$, while avoiding zero-measure sets. Several solutions for this problem exist \cite{Levant2001,Levant2003,Levant2005,ding2016simple,CRUZZAVALA2017232}. under the hypothesis that the bounds $\gamma_m,\gamma_M$ and $\bar \varphi$ are known. 

In case the bounds $\gamma_m,\gamma_M$ and $\bar \varphi$ are unknown (one only assumes their existence) then it is obvious to see that  finite time stabilization is not possible by a mere state feedback and therefore, one possible alternate objective consists in achieving {\color{black} asymptotic} stabilization. This is the goal of this paper to establish such a result for System \eqref{u.l.s.} and we provide next a precise definition of {\color{black} asymptotic}  stabilization.
\begin{defn}\label{def1} If $r,p$ are positive integers and $f:\mathbb{R}^r\times \mathbb{R}^p\rightrightarrows \mathbb{R}^r$ is a continuous differential equation, we say that the system $\dot z=f(z,u)$ is {\color{black} asymptotic} stabilizable if, 
%for every $\varepsilon>0$, 
{\color{black} there exists a continuous controller $u=U(\cdot,t)$ such that every trajectory of the closed-loop system $\dot z=f(z,U(\cdot,t))$ 
%enters the open ball of radius $\vep$ centered at the origin of $\mathbb{R}^r$ and eventually remains there. 
tends to the origin as $t$ tends to infinity.}
\end{defn}
The main result of that paper consists of designing controllers which {\color{black} asymptotically} stabilize System \eqref{u.l.s.} {\emph{independently} }of the positive bounds $\bar \varphi$, $\gamma_m$ and $\gamma_M$, i.e., 
%for every $\vep>0$, 
{\color{black} the controllers $u=U(\cdot,t)$ which asymptotically stabilize System \eqref{u.l.s.} does not depend on the bounds $\bar \varphi$, $\gamma_m$ and $\gamma_M$.}

We next recall the following definition needed in the sequel.  
\begin{defn}\label{homogeneity} ({\bf Homogeneity.} cf. \cite{Levant_Springer}.)
If $r,m$ are positive integers, a function $f:\mathbb{R}^r\rightarrow \mathbb{R}^m$ (or a differential inclusion $F:\mathbb{R}^r\rightrightarrows \mathbb{R}^m$ respectively) is said to be {\it homogeneous of degree $q\in\mathbb{R}$ with respect to the family of dilations $\delta_\vep(z)$}, $\vep>0$, defined by 
$$
\delta_\vep(z)=(z_1,\cdots,z_r)\mapsto (\vep^{p_1}z_1,\cdots,\vep^{p_r}z_r),
$$
where $p_1,\cdots,p_r$ are positive real numbers (the weights), if for every positive $\vep$ and $z\in\mathbb{R}^r$, one has $f(\delta_\vep(z))=\vep^qf(z)$  $\big(F(\delta_\vep(z))=\vep^q\delta_\vep(F(z))\ respectively\big)$.
\end{defn}

The following notations will be used throughout the paper.
%If $x\in\mathbb{R}$, we denote by $[x]$ the integer part of $x$ i.e., the smallest integer not greater than $x$. 
We define the function $ \hbox{sgn}$ as the multivalued function defined on $\mathbb{R}$ by $ \hbox{sgn}(x)=\frac x{\vert x\vert}$ for $x\neq 0$ and $ \hbox{sgn}(0)=[-1,1]$. Similarly, for every $a\geq 0$ and $x\in \mathbb{R}$, we use $\lf x\rr^a$ to denote $\left| x \right|^a  \hbox{sgn}(x)$. Note that $\lf \cdot\rr^a$ is a  continuous function for $a>0$ and of class $C^1$ with derivative equal to $a\left| \cdot \right|^{a-1}$ for $a\geq 1$. Moreover, for every positive integer $r$, we use $J_r$ to denote the $r$-th Jordan block, i.e., the $r\times r$ matrix whose $(i,j)$-coefficient is equal to $1$ if $i=j-1$ and zero otherwise. Finally, we use $e_r$ to denote the vector of $\mathbb{R}^r$ equal to $(0\ \cdots\ 0\ 1)^T$.

\subsection{{\color{black}Time Varying} Higher Order Sliding Mode Controller}\label{adaptatif}
We first define the system under study and provide parameters used later on. 
\begin{defn}\label{def0} {\it Let $r$ be a positive integer. The $r$-th order chain of integrator
$(CI)_r$ is the single-input control system given by 
\beqnum \label{pure_int}
(CI)_r\ \ \ \dot z=J_rz+u e_r,
\eeqnum
with $z=(z_1,\cdots,z_r)^T\in\mathbb{R}^r$ and $u\in \mathbb{R}$.
For $\kappa<0$ and $p>0$ with $p+(r+1)\kappa\in [0,1)$, set $p_i := p + (i-1)\kappa, \ 1\leq i\leq r+2$. For $\varepsilon>0$, let $\delta_\vep:\R^r\rightarrow \R^r$ be the family of dilations associated with $\left( p_1, \cdots , p_r \right)$.}
\end{defn}
In the spirit of \cite{Harmouche_CDC12}, we put forwards geometric conditions on certain stabilizing feedbacks $u_0(\cdot)$ for $(CI)_r$ and corresponding Lyapunov functions $V$. These conditions will be instrumental for the latter developments. 

Our construction of the feedback for practical stabilization relies on the following result. 
%{\color{blue}
%!!!!!!!!!!!!!!!!!!!!!!!!!!!!!!!!!!!!!!
%
%Cette reference concerne le CDC harmouche, baghdadi, laghrouche ainsi que nos papier avec la condition geometrique
%
%!!!!!!!!!!!!!!!!!!!!!!!!!!!!!!!
%}
\begin{Assumption}\label{theo1}

Let $r$ be a positive integer. There exists a feedback law 
$u_0:\R^r\rightarrow\R$ %continuous at $0$ with $u_0(0)=0$, 
homogeneous  with  respect to $(\delta_\vep)_{\vep>0}$ such that the closed-loop system $\dot z=J_rz+u_0(z)e_r$ is finite time globally asymptotically stable with respect to the origin and the following conditions hold true:
\begin{description}
	\item[$(i)$] the function $z\mapsto J_rz+u_0(z)e_r$ is homogeneous of degree $\kappa$ with respect to $(\delta_\vep)_{\vep>0}$ and there exists a continuous positive definite function $V\textcolor{black}{(z)}:\R^r\rightarrow \R_+$, $C^1$ except at the origin, homogeneous of positive degree with respect to $(\delta_\vep)_{\vep>0}$ such that there exists $c>0$ and $\alpha\in (0,1)$ for which the time derivative of $V\textcolor{black}{(z)}$ along non trivial trajectories of $\dot z=J_rz+u_0(z)e_r$ verifies 
	\beqnum \label{est-V1}
	\dot V\textcolor{black}{(z)} \le -cV\textcolor{black}{(z)}^\alpha;
	\eeqnum
		\item[$(ii)$] the function $z\mapsto u_0(z)\frac{\partial V\textcolor{black}{(z)}}{\partial z_r } $ is non positive over $\mathbb{R}^r$ and, for every non zero $z\in \mathbb{R}^r$ verifying $u_0(z)=0$, one has $\frac{\partial V\textcolor{black}{(z)}}{\partial z_r } =0$. As a consequence function $z\mapsto  \hbox{sgn}(u_0(z))\frac{\partial V\textcolor{black}{(z)}}{\partial z_r }$
		is well-defined over $\mathbb{R}^r\setminus\{0\}$ and non positive.
\end{description}
\end{Assumption}
%!!!!!!!!!!!!!!!!!!!!!!!!!!!!!!!!!!!!!!
%
%Cette reference concerne le CDC harmouche, baghdadi, laghrouche ainsi que nos papier avec la condition geometrique
%
%!!!!!!!!!!!!!!!!!!!!!!!!!!!!!!!
%}
\begin{rem}
Item $(i)$ of the above proposition is classical, see for instance \cite{Hong,Huang2005,Zavala_2014}. Item $(ii)$ considers a geometric condition on controllers verifying Item $(i)$, which was introduced in \cite{Harmouche_CDC12} and used in \cite{CHL15}. This geometric condition is indeed satisfied, for instance by Hong's controller, see \cite{CHL15} for other examples.
\end{rem}
%\begin{rem}\label{rem:deg1}{\it If the feedback $u_0$ and the positive function $V_1$ verify the conditions of Theorem~\ref{theo1}, then necessarily, $u_0$, $V_1$ and $\partial_rV_1$ are homogeneous with respect to  to $(\delta_\vep)_{\vep>0}$ of degree $p_{r+1}$, $2p_{r+1}$ and $p+(r+1)\kappa$ respectively. Moreover, $\alpha$ must be equal to $1+\kappa/(2p_{r+1})\geq 1/2$. Notice that $\partial_rV_1$ is of zero homogeneity degree if $p+(r+1)\kappa=0$.}
%\end{rem}

Regarding our problem, we consider, for every $\vep>0$ the following controller:
\beqnum \label{eq_v}
\textcolor{black}{u}(z,z_0,t) = g(\vert u_0(z)\vert) u_0(z) +k\, \hbox{sgn}\big( u_0(z) )\textcolor{black}{\hat \varphi \big(t,V(z_0)\big)},
\eeqnum
where  $u_0$ and $V\textcolor{black}{(z)}$ are provided by Assumption~\ref{theo1}, $g:\mathbb{R}_+\rightarrow \mathbb{R}_+^*$ is an arbitrary increasing $C^1$ function tending to infinity as with $x$ tends to infinity and 
the {\color{black}time varying function $\hat \varphi$} is defined later.
\textcolor{black}{ 
To proceed, let us first introduce, for every positive real $\xi$, the function $F_\xi$ defined on $[0,\xi)$ by $F_\xi(x)=\frac{\xi}{\xi-x}$ for $0<x<\xi$. Moreover, consider 
$\eta:\mathbb{R}_+\rightarrow \mathbb{R}_+^*$, a 
non increasing $C^1$ function %tending to zero as  $t$ tends to infinity 
so that $\eta(0)=\varepsilon>0$ and $-\dot\eta\leq M\eta$, where $\varepsilon>0$ and  $M\geq 0$ are chosen so that $M\varepsilon^{1-\alpha}< c$. Moreover, we define the following family of time varying domains, for $t\geq 0$,
\beqnum\label{eq:domain}
D(t)
=\{z\in\mathbb{R}^r\ \mid\ V(z)\leq \eta(t)\}.
\eeqnum
Assume first that $V(z_0)\leq \varepsilon/2$. Then $\hat \varphi \big(t,V(z_0))$ is defined as $F_{\eta(t)}(V(z(t)))$ for any trajectory $z(\cdot)$ of System \eqref{u.l.s.} closed-looped by the feedback control law \eqref{eq_v} and starting at $z_0$, and so, as long as $z(\cdot)$ is defined. (Note that in this case there exists a non trivial interval of existence of such solutions since $F_\eta$ is continuous.)  In case  $V(z_0)> \varepsilon/2$, then $\hat \varphi \big(t,V(z_0))=t$ as long as any trajectory $z(\cdot)$ of System \eqref{u.l.s.} closed-looped by the feedback control law \eqref{eq_v} and starting at $z_0$ is defined and verifies $V(z(t))> \eta(t)/2$. If there exists a first time 
$\bar{t}(V(z_0))>0$ such that $V(z(t))=\eta(t)/2$, then 
$\hat \varphi \big(t,V(z_0))=F_{\eta(t)}(V(z(t)))$ for $t\geq \bar{t}(V(z_0))$, and so, as long as $z(\cdot)$ is defined. Note that $\bar{t}(V(z_0))$ actually also depends on the trajectory $z(\cdot)$ since the latter may not be unique. 
Hence, as long as a trajectory $z(\cdot)$ of the closed-loop system and starting at $z_0$ is defined, the {\color{black}time varying} function $\hat \varphi \big(t,V(z_0))$ is given by
\beqnum
\hat \varphi(t,x) &=& \left\{
\begin{array}{ccc}
t, & \text{ if } 0\leq t< \bar{t}(V(z_0)) ,& \\
F_{\eta(t)}(x),  &\text{ if }t\geq \bar{t}(V(z_0)),& 
\end{array}\label{def-fi}
\right.
\eeqnum
with the convention that $\bar{t}(V(z_0))=0$ if $V(z_0)\leq \varepsilon/2$
and $\bar{t}(V(z_0))=\infty$ if $z(\cdot)$ is defined  and verifies $V(z(t))>\eta(t)$ for all non negative times $t$.
%where $\eta:\mathbb{R}_+\rightarrow \mathbb{R}_+^*$ is a 
%strictly decreasing $C^1$ function tending to zero as  $t$ 
%tends to infinity so that $\eta(0)=1$ and $-\dot\eta\leq m\eta$ for some 
%constant $m>0$ and $F_\xi(x)=\frac{\xi}{\xi-x}$ for $0<x<\xi$}. 
\\
Here the positive function $g$ and the positive constant $k$ 
 are  gain parameter.
 }
% will be used to define the arbitrarily small neighborhood where trajectories of \eqref{u.l.s.} with feedback control law \eqref{eq_v} will eventually end up.   

The following theorem provides the main result for the {\color{black}time varying} controller ${\textcolor{black}u}$.
%%%%%%%%%%%%%%%%%%%%%%%%%%%%%%%%%%%%%%%%%%%%%%%%%%%%%%%%%%%%%%%%%%%%%%%%%%%%%%%%%%%%%%%%%%%%%%%%%%%%%%%%%%%%%%%%%%%%%%%%%%%%%
\begin{montheo}\label{theorem_adaptatif}
Let $r$ be a positive integer and System \eqref{u.l.s.} be the perturbed $r$-chain of integrators with unknown bounds $\gamma_m,\gamma_M$ and $\bar{\varphi}$. Let 
%$\vep>0$ and 
$u_0,V:\mathbb{R}^r\rightarrow \mathbb{R}_+$ be the feedback law and the continuous positive definite function defined respectively in Assumption~\ref{theo1}. For every 
$z_0\in\mathbb{R}^r$, consider any trajectory $z(\cdot)$ of System \eqref{u.l.s.} closed by the feedback control law \eqref{eq_v} verifying $z(0)=z_0$. Then, $z(\cdot)$ is defined for all non negative times, there exists a first time $\bar{t}(V(z_0))$ for which 
$V(z(t))\leq \eta(t)/2$ at $t=\bar{t}(V(z_0))$ and $z(t)\in D(t)$
%$V(z(t))< \eta(t)$ 
for all $t>\bar{t}(V(z_0))$.
%\beqnum\label{mainR}
%\textcolor{magenta}{\lim_{t\rightarrow \infty}z(t)=0}.% \leq \max\big(0,\overline{V}_\vep\big),
%\eeqnum
%where $\overline{V}_\vep:=\vep(1-\frac1{\bar\Phi})$, with $\bar \Phi := \frac1{k\gamma_m}\left( \bar \varphi-h_m\right)$, where $h_m=\min\big(0,\min_{x\geq 0}(\gamma_mg(x)-1)x\big)$.
%\\
%  Moreover, if $l>0$ is the homogeneity degree of $V$, then we have the following result: for every 
%\beqnum\label{precision}
%\limsup_{t\rightarrow \infty}\vert z_i(t)\vert\leq C_i\vep^{p_i/l}, \quad 1\leq i\leq r,
%\eeqnum
%for some constants $C_i$, $1\leq i\leq r$.
\end{montheo}
%\begin{rem}
%\textcolor{magenta}{ In Remark~\ref{remarkE} below, we give more precise indications on how the convergence to the origine insured by \eqref{mainR} actually occurs.}
%\end{rem}

%\begin{rem}
%\textcolor{black}{Note that, in the case where the lower bound $\gamma_m$ is only known whereas the upper bounds $\gamma_M$ and $\bar{\varphi}$ are not known, the function $g(\vert u_0(z)\vert)$ and the constant $k$ can be chosen as $1/\gamma_m$, and thus the control law \eqref{eq_v} can be simplified as follows:
%\beqnum \label{eq_vknowngamma}
%u_\vep(z,t) = \frac{1}{\gamma_m} \Big( u_0(z) +\, \hbox{sgn}\big( u_0(z) )\hat \varphi_\vep (t,V(z)\big)\Big). 
%\eeqnum }
%\end{rem}
\subsection{Proof of Theorem \ref{theorem_adaptatif}}
We refer to $(S)$ as the closed-loop system defined by  \eqref{u.l.s.} and \eqref{eq_v}. 
The first issue we address is the existence of trajectories of $(S)$ starting at any initial condition $z_0\in\mathbb{R}^r$. Such an existence follows from the fact that the application $\mathbb{R}_+\times \mathbb{R}^r\rightarrow \mathbb{R}$, $(t,z)\mapsto \textcolor{black}{\hat \varphi (t,V(z))}$ is continuous. 

We next show that every trajectory of $(S)$ is defined for all non negative times. For that purpose, consider a non trivial trajectory $z(\cdot)$ and let $I_{z(\cdot)}$ be its (non trivial) domain of definition.  We obtain the following inequality for the time derivative of  $V(z(\cdot))$ on  $I_{z(\cdot)}$ by using Items $(i)$ and $(ii)$ of Assumption~\ref{theo1}. For a.e. $t\in I_{z(\cdot)}$, one gets

\beqnum \label{dot_V1}
	\dot V \hspace{-2mm}&=&  \hspace{-2mm}\frac{\partial V}{\partial z_1 }z_2 + ... + \frac{\partial V}{\partial z_r }  \left( \gamma  \left[ g(\vert u_0\vert)  u_0 +  k\hbox{sgn} (u_0) \textcolor{black}{\hat \varphi}  \right] + \varphi \right),\\
		%&=& \frac{\partial V_1}{\partial z_1 } z_2 + ... + \frac{\partial V_1}{\partial z_r } u_0+ \frac{\partial V_1} {\partial z_r }\left( -u_0 \!+\! A \gamma u_0 \!+\! \gamma \delta \left\lfloor u_0 \right\rceil ^2  \!+\! \gamma sgn (u_0) \hat \varphi   \!+\!  \varphi \right),\\
&\le& -c V^\alpha \!-\! \!\left|\frac{\partial V}{\partial z_r } \right| \!\! \Big( \! (\gamma_m g(\vert u_0\vert)-1) |u_0|  \!+\! k\gamma_m \textcolor{black}{\hat \varphi} \!-\! \bar \varphi \Big),\\
%&=&-c V_1 ^\alpha - \left|\frac{\partial V_1} {\partial z_r } \right| \left[ \gamma_m \delta \left( |u_0| + \frac{A \gamma_m -1}{2\gamma_m \delta}  \right)^2 \right]- \left|\frac{\partial V_1}{\partial z_r } \right|\left[ - \left( \bar \varphi + \frac{\left( A \gamma _m -1\right)^2}{4\gamma_m \delta} \right) + \gamma_m \hat \varphi \right],\\
&\le& -c V ^\alpha - k\gamma_m\left|\frac{\partial V}{\partial z_r } \right| \left( \textcolor{black}{\hat \varphi} - \bar \Phi\right),
\eeqnum
\textcolor{black}{with $\bar \Phi:= \frac1{k\gamma_m}\left( \bar \varphi-h_m\right)$, where 
$$
h_m=\min\big(0,\min_{x\geq 0}x(\gamma_mg(x)-1)\big).
$$}
We thus have the differential inequality a.e. for $t\in I_{z(\cdot)}$
\beqnum\label{dot_V1-1}
\dot V \leq -c V^\alpha +C_1\bar \varphi V^{p_{r+2}},
\eeqnum
where $C_1$ is a positive constant independent of the trajectory $z(\cdot)$. Since $p_{r+2}\in [0,1)$, it is therefore immediate to deduce that there is no blow-up in finite time and thus $I_{z(\cdot)}=\mathbb{R}_+$.  \\
\textcolor{black}{The next step consists in showing the existence of a finite time $\bar{t}(V(z_0))$ and, for that purpose, we can assume with no loss of generality that $V(z_0)> \varepsilon/2$. The corresponding trajectory is defined as long as $V(z(t))>\eta(t)/2$, since in that case, the growth of the right-hand side of \eqref{u.l.s.} is sublinear with respect to the state variable $z$. Arguing by contradiction, one gets that $V(z(t))>\eta(t)/2$ for all non negative times, and hence $\dot V\leq -c V ^\alpha$ for $t\geq \bar \Phi$. The latter inequality yields convergence to the origine in finite time, which is a contradiction. Then, the existence of a finite $\bar{t}(V(z_0))$ is established. 
\\
One is left to prove that  $z(t)\in D(t)$ for $t>\bar{t}(V(z_0))$. Since ${\hat \varphi}=F_{\eta}>1$
on that time interval, one gets finite time convergence if $\bar \Phi\leq 1$, whatever the choice of $\eta$ is. Assume that $\bar \Phi>1$. For every $t>\bar{t}(V(z_0))$, let $V_*(t)$ be the unique solution in $(0,\eta(t)$ of the equation $F_{\eta(t)}(V_*(t))=\bar \Phi$, i.e., $V_*(t)=(1-1/\bar\Phi)\eta(t)$. Note that if $V(z(t))\geq V_*(t)$, then $\dot V\leq -c V ^\alpha$.
We will actually prove that 
\beqnum\label{barPhi}
V(z(t))\leq \max(\frac12,(1-1/\bar\Phi))\eta(t), \hbox{ for }t>\bar{t}(V(z_0)).
\eeqnum
and, if $\bar \Phi\geq 2$, then there exists $t'\geq \bar{t}(V(z_0))$ such that 
$V(z(t))\leq (1-1/\bar\Phi)\eta(t)$ for $t>t'$. Assume first that $V(z(t))\geq (1-1/\bar\Phi)\eta(t)$ at 
$t=\bar{t}(V(z_0))$. Then, $\dot V\leq -c V ^\alpha$ in a right neighborhood of $\bar{t}(V(z_0))$
and, by the previous argument by contradiction, there must exists a first time $t'\geq \bar{t}(V(z_0))$
so that $V(z(t))=(1-1/\bar\Phi)\eta(t)$ at $t=t'$. The time derivative of $t\mapsto V(z(t))-(1-1/\bar\Phi)\eta(t)$ at $t=t'$ is less than or equal to 
$$
-cV_*(t')^{\alpha}+MV_*(t')\leq -V_*(t')^{\alpha}(c-MV_*(t')^{1-\alpha})
$$
which is negative since $V_*(t')<\eta(0)=\varepsilon$. Note also that the previous inequality holds true at every time $s$ so that $V(z(s))=(1-1/\bar\Phi)\eta(s)$. 
Then, the zeros of $V-(1-1/\bar\Phi)\eta$ on $[T_m,\infty)$ are isolated and  $V-(1-1/\bar\Phi)\eta>0$ ($V-(1-1/\bar\Phi)\eta<0$ resp.) in a left (right resp.) neighborhood of $s$.
Hence $t'$ is the unique zero of $V-(1-1/\bar\Phi)\eta$ on $[\bar{t}(V(z_0)),\infty)$ on $[\bar\Phi)\eta(t),\infty)$ and the claim is proved. 
\\
Finally assume that $V(z(t))< (1-1/\bar\Phi)\eta(t)$ at $t=\bar{t}(V(z_0))$. By the previous computations, one gets that $V(z(t))< (1-1/\bar\Phi)\eta(t)$ for all $t\geq \bar{t}(V(z_0))$.
This concludes the proof of Theorem~\ref{theorem_adaptatif}.
}

 \begin{rem}\label{remarkg}
At the light of the above argument, one can see that, if the bounds of the incertainties are known, then one can choose the gain parameter $k$ in such a way that $ \bar\Phi\leq 1$ and hence  get finite time convergence to zero. In that way, our controller provides yet another  finite time stabilizer of the perturbed integrator with known bounds on the perturbations. %Moreover, in case the controller $u_0$ is bounded, the assumption that $\lim_{x\rightarrow +\infty}g(x)=+\infty$ can be removed.
\end{rem}

\textcolor{black}{
 \begin{rem}\label{remarkE} From \eqref{barPhi} and the choice $\eta(t)=\varepsilon e^{-Mt}$, one deduces that $V(z(t))<\varepsilon e^{-Mt}$ for $t\geq \bar{t}(V(z_0))$. In particular, the trajectory converges to the origin with the exponential rate $M>0$ which can be chosen at will.
 Moreover, it is easy to provide with the help of \eqref{dot_V1} upper bounds on $\bar{t}(V(z_0))$
 in terms of $\bar\Phi$ and $V(z_0)$.
%\\
%One must however estimate the time $\bar{t}(V(z_0))$ needed to get \eqref{barPhi}. To simplify the discussion, we hence assume $ \bar\Phi>1$ and $m$ large. The last hypothesis implies that $t_m<1$ and therefore $T_m= \bar\Phi$. From \eqref{dot_V1-1}, one gets the upper bound $V(z( \bar\Phi))\leq V(z(0))+f(\bar\Phi)$ for some \emph{explicit} function $f$. Then, 
%by the end of the argument of Theorem~\ref{theorem_adaptatif}, we have either $t_*= \bar\Phi$ or $t_*$ admits an \emph{explicit} upper bound in terms of $V(z(0))$ and $f(\bar\Phi)$. 
\\
To be complete, one should emphasize that the result in Theorem~\ref{theorem_adaptatif} does guarantee that a trajectory entering in the neighborhood $\{z,\ V(z)\leq \eta(t)/2\}$ for the first time at $t=\bar{t}(V(z_0))$ will always remain in the larger neighborhood $\{z,\ V(z)< \eta(t)\}$ for $t\geq \bar{t}(V(z_0))$. 
%It only insures that, past the time $t_*$ it enters in it and \emph{remains in it}. 
%One can see $t_*$ as a delay before entering in the neighborhoods $\{z,\ V(z)\leq \eta(t)\}$.
%This is a limitation of our controller since for a perturbation made of two non trivial pieces consecutive in time, with the second one much stronger (in amplitude) than the first one, trajectories will first converge to the neighborhoods $\{z,\ V(z)\leq \eta(t)\}$ (hence close to the origin) and after the second piece of the perturbation occurs, they will go back to such neighborhoods only after a time $t_*$ proportional to the second level of the perturbation. 
\end{rem}
}
\textcolor{black}{
 \begin{rem}\label{remarkE} 
 One other way to diminish the delay time $\bar{t}(V(z_0))$ needed to enter into the neighborhoods $\{z,\ V(z)\leq \eta(t)\}$ consists in replacing the time $t$ in the definition of $\hat{\phi}$ given in \eqref{def-fi} by an increasing function $l(t)$ tending to infinity faster than a linear one. In that manner, inequalities such as $t> \bar\Phi$ is replaced by $l(t)> \bar\Phi$.
 The price to pay will be an larger upper bound for the gains.
 \end{rem}
}
\begin{rem}
Among different controllers that can fulfill Assumption~\ref{theo1}, Hong's controller \cite{Hong} can be used. This controller is defined as follows:
\\ Let $\kappa<0$ and $l_1,\cdots,l_r$ positive real numbers. For $z=(z_1,\cdots,z_r)$, we define $u_0 = v_r$ for $ i=0,...,r-1$:
\beqnum\label{Hongfunction}
\ v_0=0,\,
v_{i+1} = -l_{i+1} \lfloor\lfloor z_{i+1} \rceil^{\beta_i } - \lfloor v_i \rceil^{\beta_i } \rceil^{(\alpha_{i+1}/\beta_i)},
\eeqnum
where $p_i = 1+(i-1)\kappa, \; \beta_0 =p_2,\; (\beta_i + 1)p_{i+1} = \beta_0 + 1 > 0 \;\text{and} \; \alpha_i=\frac{p_{i+1}}{p_i}$.\\
Now, let $\psi_r(z_1, \cdots,z_r)=\lfloor z_{r} \rceil^{\beta_{r-1} } - \lfloor v_{r-1} \rceil^{\beta_{r-1} }$. Then according to \cite{harmouche2017stabilisation}, if $\kappa=-1/r$, $u_0= -l_r\,  \hbox{sgn}\Big(\psi_r(z_1, \cdots,z_r) \Big)$ is bounded and the {\color{black}time varying} controller $u$ can be written as
\beqnum\label{eq:robustcont}
u(z,t) = g(\vert u_0(z)\vert) u_0(z) +k\, \hbox{sgn}\big( u_0(z) )\hat \varphi (t,V(z)\big)\\ \quad \quad \quad =- \Big(g(\vert u_0(z)\vert) l_r +k \hat \varphi (t,V(z)\big)\,\Big) \hbox{sgn}\big(\psi_r(z_1, \cdots,z_r) \big)
\eeqnum
Here the assumption that $\lim_{x\rightarrow +\infty}g(x)=+\infty$ can be removed. \textcolor{black}{Indeed, in view of \eqref{dot_V1} and taking into account that $|u_0|=l_r$, one can remove this assumption and the result remains valid}. Moreover, in the case when the bounds of the uncertainties are known, the functions $g$ and $\hat \varphi$ can be chosen constants, which leads to homogeneous controller $u$. %Note that the above controller scheme is similar to the one proposed in \cite{ding2016simple,cruz2017homogeneous}. 

  \end{rem}
\subsection{Asymptotic bounds for the controller $u$} 
%and the convergence time to the neighborhood ${\cal{V}}_\vep$}
One deduces from Theorem~\ref{theorem_adaptatif} the following 
 \textcolor{black}{result which
 provides an asymptotic upper bound for the controller $u$ that does not exhibit any overestimation.} 

 \textcolor{black}{
\begin{lemme}
%Define $U_0(\vep):=\max_{\{z\, \vert\,  V_1(z) \le \varepsilon \}} |u_0(z)|$. Then, t
%For $\vep>0$, t
The controller $u$ defined in \eqref{eq_v} verifies the following asymptotic upper bound, which is uniform with respect to trajectories of the closed-loop system $(S)$: 
\beqnum
\limsup_{t\rightarrow \infty}{|u|} \le M(u_0)+k\max(1,\bar\Phi),
%\left\{
%\begin{array}{ccc}
%g(l_r)l_r+k\max(1,\bar\Phi),  \text{ if $u_0$ is homogeneous of degree zero}  ,& \\
%k\max(1,\bar\Phi),  \text{ if $u_0$ is homogeneous of positive degree},& 
%\end{array}
%\right.
\eeqnum
where $M(u_0)=g(l_r)l_r$ if $u_0$ is homogeneous of degree zero, where $l_r$ is the supremum of $\vert u_0\vert$ over $\mathbb{R}^r$, and $M(u_0)=0$ if $u_0$ is homogeneous of positive degree.
%where 
%$U_0(\xi):=\max_{\{z\, \vert\,  V(z) \le \xi \}} |u_0(z)|$ for $\xi\geq 0$.
\end{lemme}
 \begin{proof}
 It is enough to prove that $\hat\varphi(t,V(z(t)))\leq \max(1,\bar\Phi)$ for $t$ large enough. Indeed, a careful examination of the argument of Theorem \ref{theorem_adaptatif} shows that either $\bar\Phi\leq 1$, in which case one has convergence to zero in finite time and $\hat\varphi=1$ for $t$ large enough,
 or $\bar\Phi>1$, in which case \eqref{barPhi} holds true and then $\hat\varphi\leq \bar\Phi$ for $t$ large enough.  
 \end{proof}}

%\textcolor{blue}{ Inutile }\textcolor{magenta}{ For $\vep>0$, define  the open neighborhood ${\cal{V}}_\vep$ of the origin as the set of points $z\in\mathbb{R}^r$ such that $V(z)<\vep$. Our second result  provides an asymptotic upper bound for the time needed by any trajectory
%of the closed-loop system $(S)$ to eventually enter ${\cal{V}}_\vep$ and remain  inside.
%\begin{Proposition}\label{prop:temps}
%For $z_0\in\mathbb{R}^r$, the convergence time $T_{z_0}$ needed to reach the open neighborhood ${\cal{V}}_\vep$ of the origin and stay inside verifies the following bound: 
%\beqnum\label{eq:Tz0}
%T_{z_0}\leq \bar \Phi+\frac{\Big(V^{1-p_{r+2}}(z_0)+(1-p_{r+2})C_1\bar\Phi\Big)^{\frac{1-\alpha}{1-p_{r+2}}}}{c(1-\alpha)},
%\eeqnum
%where the constants $c,\alpha$ are provided by proposition ~\ref{theo1} and the constants $\beta,C_1$ are defined in Eq.~\eqref{dot_V1-1}.
%\end{Proposition}
%\begin{proof}
%For $t\leq \bar \Phi$, one deduces from Eq.~\eqref{dot_V1-1} the upper bound 
%$$
%V(z( \Phi))\leq\Big( V^{1-p_{r+2}}(z_0)+(1-p_{r+2})C_1\bar\Phi\Big)^{\frac1{1-p_{r+2}}}.
%$$ 
%For $t\geq \bar \Phi$, $V$ verifies either 
%Eq.~\eqref{est-V1} or Eq.~\eqref{eq:V_1-1}, which reduces to Eq.~\eqref{est-V1} if $z(t)\notin{\cal{V}}_\vep$. It is then clear that the right-hand side of Eq.~\eqref{eq:Tz0} is an upper bound for $T_{z_0}$. 
%\end{proof}
%}

\section{Simulation Results}\label{simulation}

Consider the following third order system
\beqnum \label{eq:systemchain3}
	\dot z_1 &=& z_2,\\
	\dot z_2 &=&  z_3,\\
	\dot z_3 &=& \varphi + \gamma u,
\eeqnum
where $\gamma$ and $\varphi$ are \textit{discontinuous} bounded uncertainties defined as
\beqnum
	\varphi = 5\,  \hbox{sgn}(\cos(t)) - 20 \sin(2t), \quad |\varphi| \le 25 ,\\
	\gamma = 3 - 2 \,  \hbox{sgn}(\sin(3t)), \quad 0 < 1 \le \gamma \le 5,
\eeqnum
and $z_0=[4,4,-4]$. 
\\
%In the following subsections, there controllers will be considered: first the robust Hong's controller in the case when the uncertainties boundaries are supposed to be known. Then, for the case when the boundaries are unknown, the proposed adaptive scheme will be used to design two adaptive controller's.
In the following subsections, two cases will be considered: first the case when the bounds of the uncertainties are known, and second when they are unknown. 
\\
\textcolor{black}{The first case is provided to show that the controller \eqref{eq:robustcont} can be considered as HOSMC due to the reason that it provides $r$-th order of asymptotic precision with respect to the
sampling step}. \textcolor{black}{While in the second case, the effectiveness of the proposed approach to force the sliding variable and its $r-1$ first derivatives to the following family of time varying domains 
\beqnum
D(t)
=\{z\in\mathbb{R}^r\ \mid\ V(z)< \eta(t)\}.
\eeqnum
is studied using the controller \eqref{eq:robustcont}.}

%there controllers will be considered: the robust Hong's controller \cite{Hong}, an adaptive version of Hong's controller \cite{Hong} based on Theorem~\ref{theorem_adaptatif} and an adaptive version of the controller proposed in \cite{mendoza2017idea} based on the same theorem. }

\subsection{Case when the bounds of uncertainties are known }\label{sub1}
In this subsection, the control parameters of $u_0$ in \eqref{eq:robustcont} are tuned to the following values
\beqnum
		l_1=1,\ l_2=2,\ l_3 = 5,\ \kappa = -1/3,
\eeqnum
the constant $k$ and the function $g$, $\hat \varphi$ in \eqref{eq:robustcont} are selected as $k={\bar \varphi \over{\gamma_m}}$, $g={1\over{\gamma_m}}$, $\hat \varphi=1$. Hence, 
\beqnum \label{eq:robsut}
u=  -{l_3 + \bar \varphi \over{\gamma_m}}\, \hbox{sgn}\Big( \psi_3(z_1,z_2,z_3)\Big) ,
\eeqnum
with $\psi_3(z_1,z_2,z_3)$
 is given explicitly in appendix~\ref{appendix:1}.

\begin{figure}[tbp]
\includegraphics[trim= 0.5cm 3.cm 0.5cm 2.0cm, clip, width=9cm]{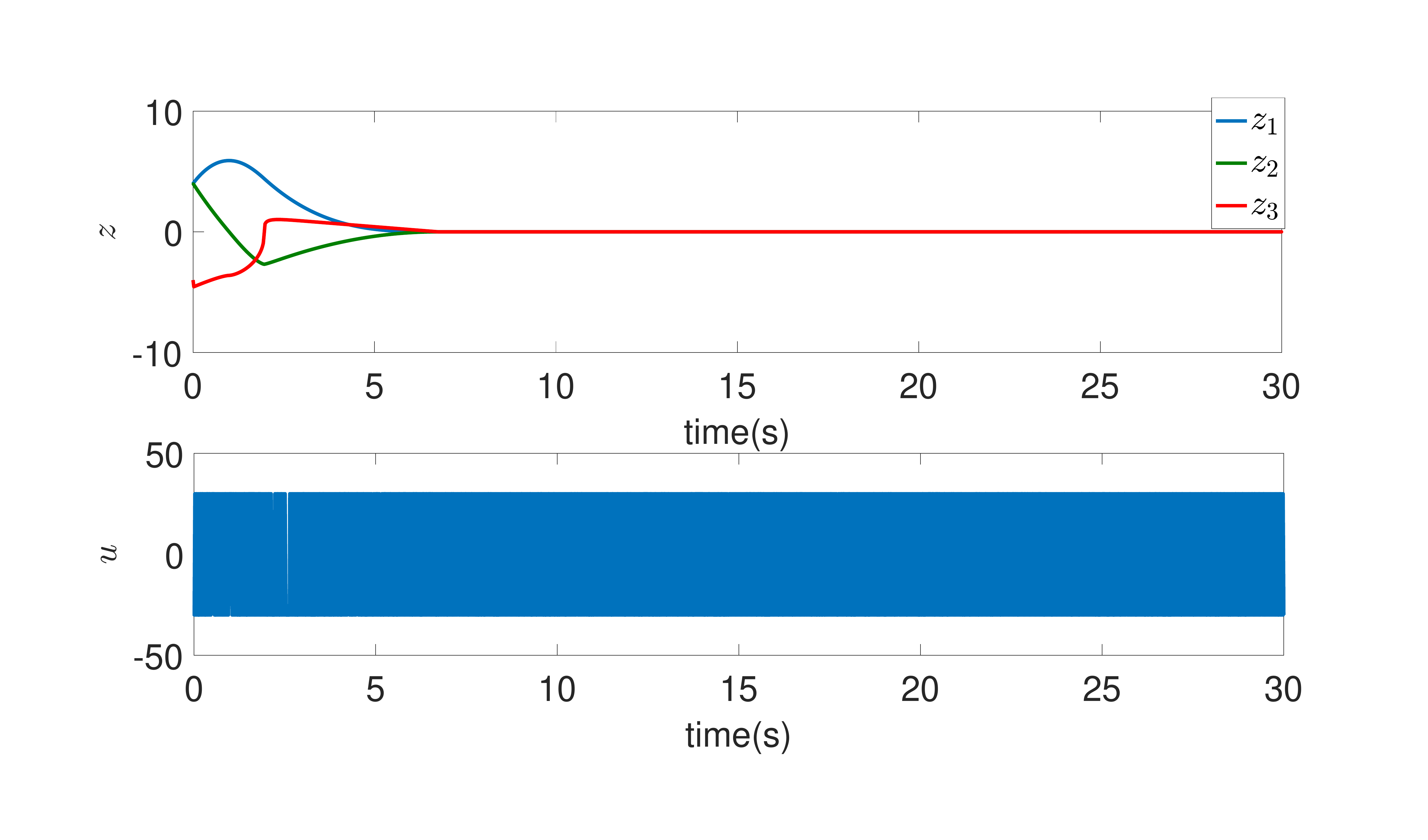}%
\caption{System \eqref{eq:systemchain3} in close loop with \eqref{eq:robsut}. }
\label{fig:robust1}
\end{figure}

\begin{figure}[tbp]
\includegraphics[trim= 0.5cm 3.cm 0.5cm 2.0cm, clip, width=9cm]{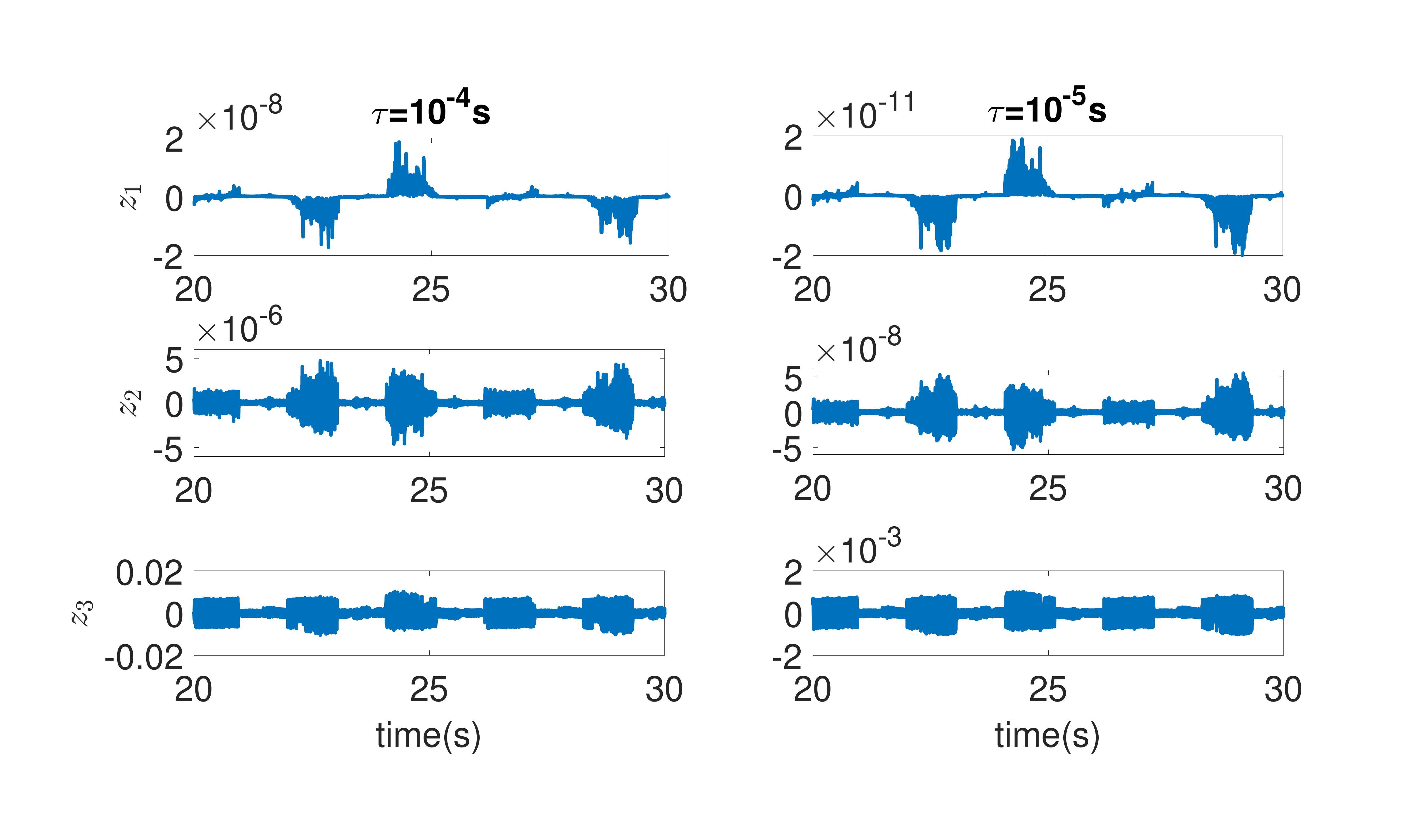}%
\caption{Accuracy of System \eqref{eq:systemchain3} in close loop with \eqref{eq:robsut} for $\tau=10^{-4}~s$ and $\tau=10^{-5}~s$. }
\label{fig:robust2}
\end{figure}

Fig~\ref{fig:robust1} shows the simulation results of system \eqref{eq:systemchain3} with controller \eqref{eq:robsut}. We can see that the states $z_1$, $z_2$, $z_3$ converge in a finite time to the origin, and the control input $u$ is discontinuous.

The system homogeneity degree is $\kappa=-1/3$ and the homogeneity weights of $z_1$, $z_2$, $z_3$ are $1$, $2/3$, $1/3$ respectively. These weights can be multiplied by 3 in order to achieve homogeneity weights $3$, $2$, $1$ with the closed loop system homogeneity degree $\kappa=-1$. This homogeneity is called 3-sliding homogeneity \cite{Levant2005}. According to these homogeneity weights, the controller \eqref{eq:robsut} provides the following accuracy for the states w.r.t sampling step $\tau$\begin{equation}  \label{eq:pres1}
|z_1| \le \lambda_1 {\tau}^3, \quad |z_2| \le \lambda_2 {\tau}^2, \quad |z_3| \le \lambda_3 {\tau},
\end{equation} 
where $\lambda_i$ are constants. By simulations shown in Fig~\ref{fig:robust2} with ${\tau}=10^{-4}s$, constants $\lambda_i$ are determined as $\lambda_1=20000$, $\lambda_2=500$, and $\lambda_3=200$. These constants have been confirmed by simulations with ${\tau}=10^{-5}s$ also shown in Fig~\ref{fig:robust2}.\\ 

\begin{figure}[tbp]
\includegraphics[trim= 0.5cm 3.cm 0.5cm 2.0cm, clip, width=9cm]{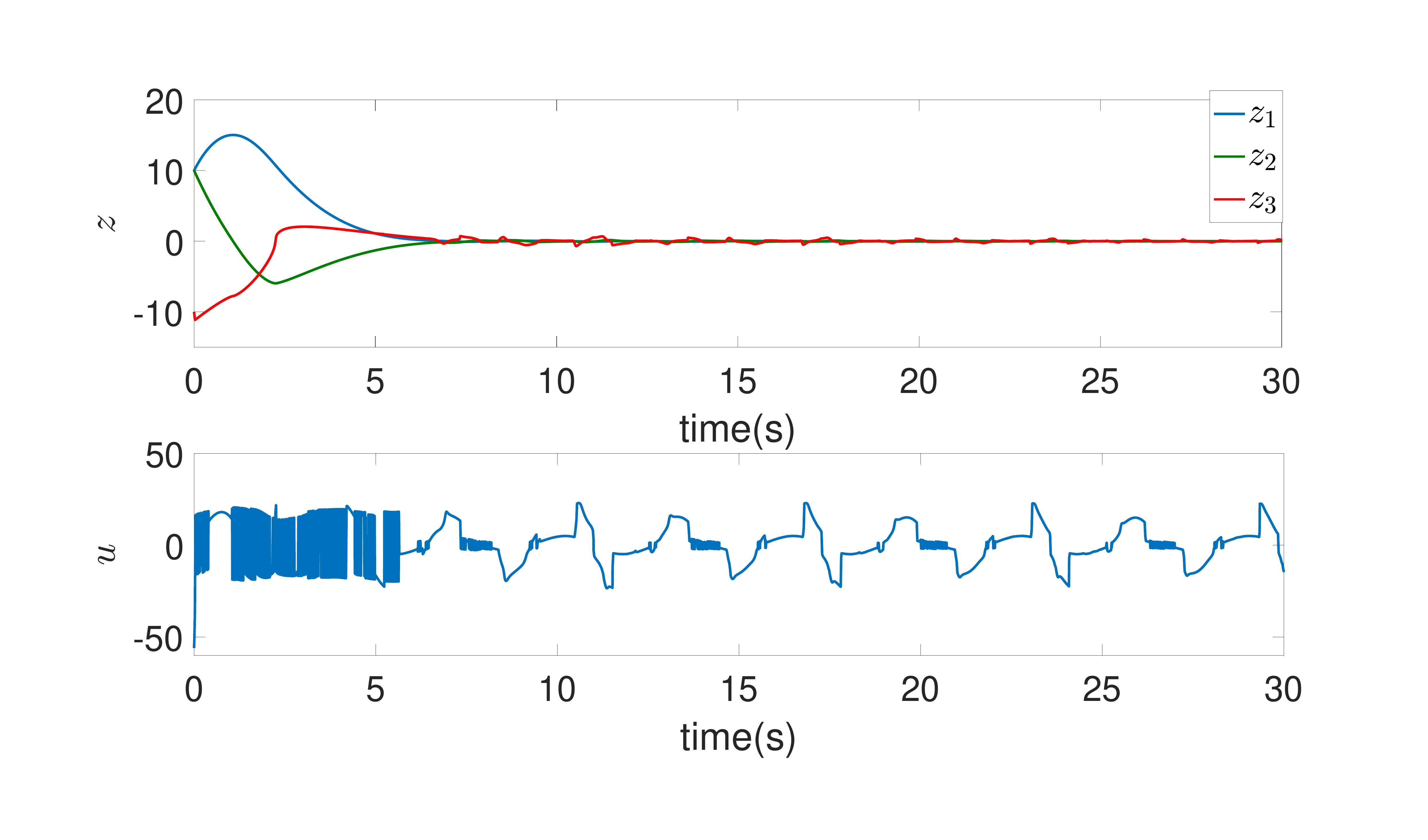}%
\caption{System \eqref{eq:systemchain3} in close loop with \eqref{eq:adapthong}. }
\label{fig:hongsystem}
\end{figure}

\subsection{Case when the bounds of uncertainties are unknown} \label{sub2}

{\color{black} In this case, the control parameters of $u_0$ are the same as in subsection~\ref{sub1}, while the constant $k$ and the functions $g$ and $\hat \varphi$ in \eqref{eq:robustcont} are selected as $k=1$, $g=1$ and $\hat \varphi$ follows \eqref{def-fi}  where $\eta(t)= e^{-0.2 t}$.} Hence, $u$ can be written as
\beqnum  \label{eq:adapthong}
u =-\Big( l_3 + \hat \varphi(t,V(z)\big)\,\Big)\, \hbox{sgn}\Big( \psi_3(z_1,z_2,z_3)\Big),
\eeqnum
with the Lyapunov Function candidate $V(z)$ is given by (see appendix~\ref{appendix:1})
\begin{align*}
\begin{split} 
V(z)  & = \frac{3}{5}{\left| {{z_1}} \right|^{\frac{5}{3}}}+ \frac{2}{5}{\left| {{z_2}} \right|^{\frac{5}{2}}} + \frac{3}{5}l_1^{\frac{5}{2}}| {{z_1}} |^{\frac{5}{3}} + l_1^{\frac{3}{2}}{z_2}{\lfloor {{z_1}}\rceil}.+ \frac{1}{5}{\left| {{z_3}} \right|^5} + 
\\& \frac{4}{5}l_2^5{\left\lfloor {{{\left\lfloor {{z_2}} \right\rceil }^{\frac{3}{2}}} + l_1^{\frac{3}{2}}{{\left\lfloor {{z_1}} \right\rceil }}} \right\rceil ^{\frac{5}{3}}} + {z_3}l_2^4\left\lfloor { {{{\left\lfloor {{z_2}} \right\rceil }^{\frac{3}{2}}} + l_1^{\frac{3}{2}}{{\left\lfloor {{z_1}} \right\rceil }}} } \right\rceil ^{\frac{4}{3}}.
\end{split} 
\end{align*}

\begin{figure}[tbp]
\includegraphics[trim= 0.5cm 3.cm 0.5cm 2.0cm, clip, width=9cm]{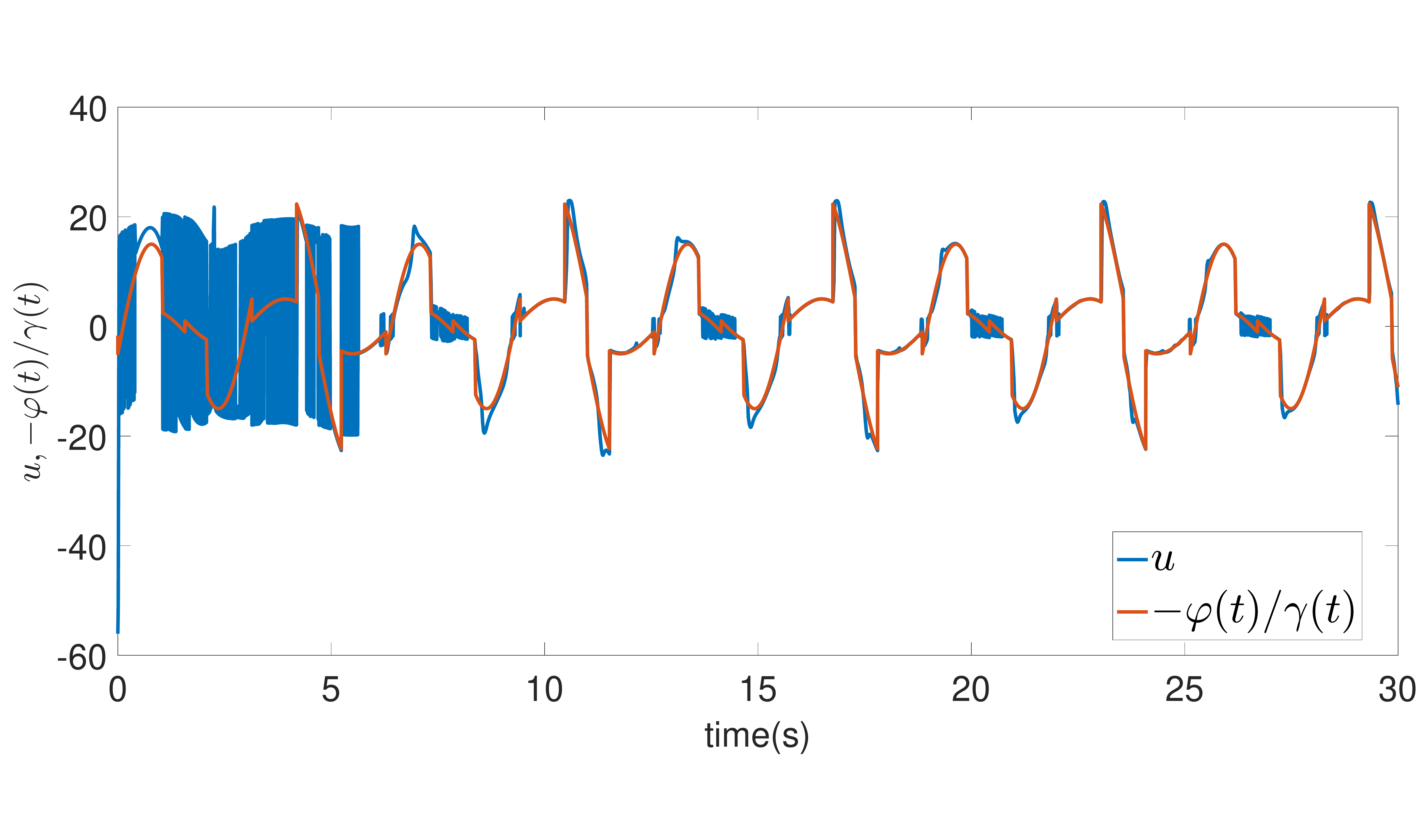}%
\caption{Evolution of the control signal and the uncertainties for System \eqref{eq:systemchain3} with \eqref{eq:adapthong}}
\label{fig:control}
\end{figure}

\begin{figure}[tbp]
\includegraphics[trim= 0.5cm 3.cm 0.5cm 2.0cm, clip, width=9cm]{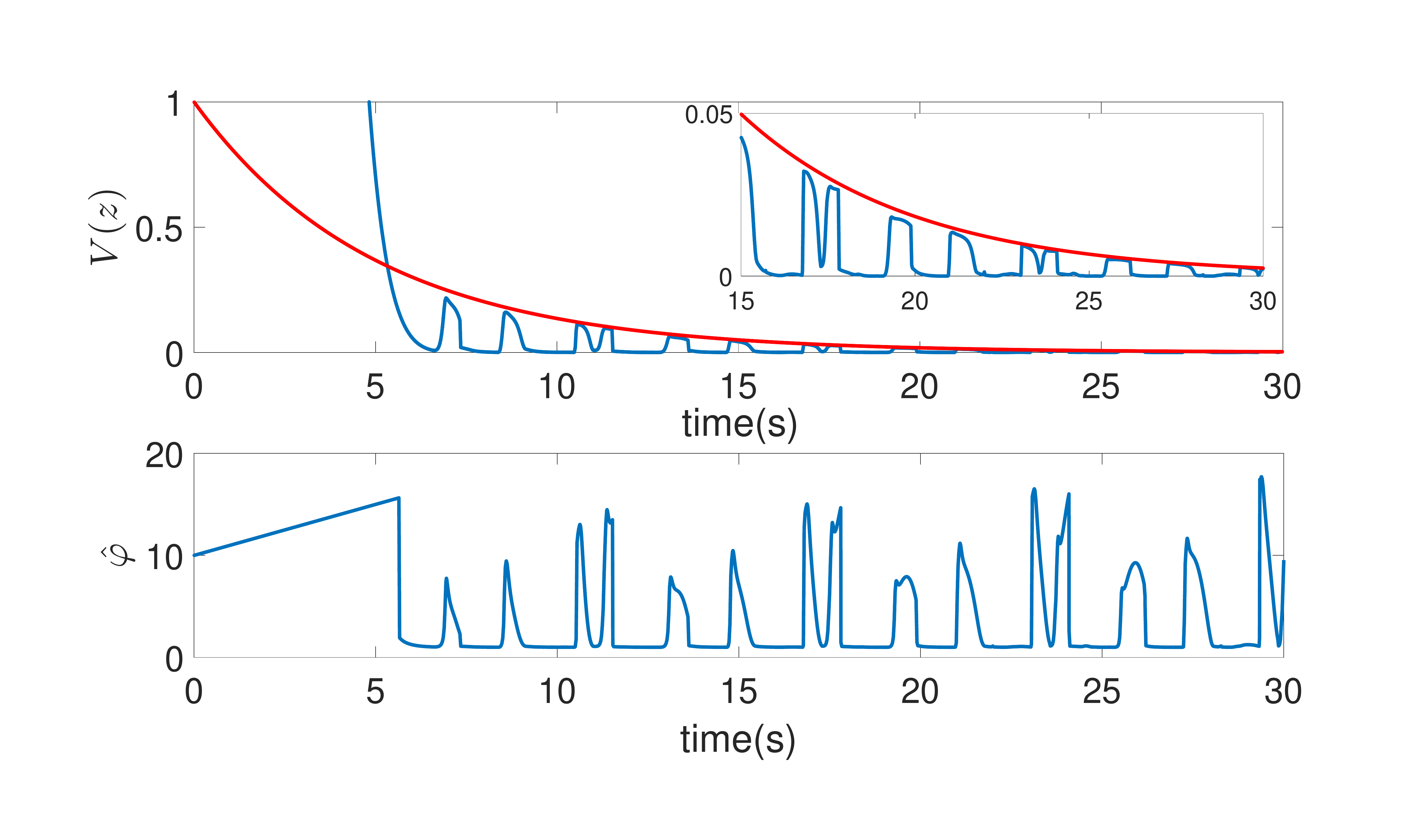}%
\caption{Evolution of the Lyapunov function $V(z)$ and the time varying function $\hat \varphi(t,V(z)\big)$ for System \eqref{eq:systemchain3} with \eqref{eq:adapthong}. }
\label{fig:lyqpunov}
\end{figure}

{\color{black} Figs~\ref{fig:hongsystem}-\ref{fig:control}-\ref{fig:lyqpunov} illustrate the simulation results of system \eqref{eq:systemchain3} with the time varying controller \eqref{eq:adapthong}. It can be noticed in Fig~\ref{fig:hongsystem} that the proposed controller provides the practical stabilization for system \eqref{eq:systemchain3}, and the control input is not continuous in general. Moreover, it is confirmed in Fig~\ref{fig:control} that the control signal is not overestimated. Indeed, after the time instant $\bar{t}(V(z_0))$ the control signal closely follows the uncertainties. On the other hand, the fulfillment of the control objective can be shown in Fig~\ref{fig:lyqpunov}, where the sliding variable and its $r-1$ first derivatives enter in finite time in the family of time varying domains $D(t)$ defined in \eqref{eq:domain} and cannot exit it anymore for larger times. It can be seen also that the time varying function $\hat \varphi(t,V(z)\big)$ increases linearly until the time instant $\bar{t}(V(z_0))$ and then starts to follow the function $F_{\eta(t)}(x)$.
}

%%%%%%%%%%%%%%%%%%%%%%%%%%%%%%%%%%%%%%%%%%%%%%%%%%%%%%%
\section{Conclusions}
\noindent \textcolor{black}{ This paper has proposed a new Lyapunov-based time varying scheme for higher-order sliding mode controller applied for a class of perturbed chain of integrators with unknown bounded uncertainties. The proposed time varying controller guarantees $(a)$ the finite time convergence to  a family of time varying open sets $(D(t))_{t\geq 0}$ decreasing to the origin as $t$ tends to infinity; $(b)$ once a trajectory enters some $D(t_*)$ at time $t_*$, it remains trapped in the $D(t)$'s, i.e., $z(t)\in D(t)$ for $t\geq t_*$. Another advantage of this time varying controller consists in the fact that the decrease to the origin of the family $(D(t))_{t\geq 0}$ can be chosen arbitrarily.} %As a result, there is no state overshoot,  the variable gain structure allows the gain to decrease. 

\section*{References}

\bibliographystyle{plain}
\bibliography{BAHOSMC_BIB}

\begin{appendices}
\section{$u_0$ and the Lyapunov function $V$ design \cite{Hong}} \label{appendix:1}

Firstly, the controller $u_0$ is determined explicitly. \\Let the system homogeneity degree $\kappa =  - \frac{1}{3}$. Then the homogeneity weights for the states $z_1$, $z_2$, $z_3$ are $1$, $\frac{2}{3}$, $\frac{1}{3}$ respectively and the constants $\beta_i$ are ${\beta _0} = \frac{2}{3}$, ${\beta _1} = \frac{3}{2}$, ${\beta _2} = 4$.
According to \cite{Hong}, $u_0=v_3$ with
\begin{align} \label{eq:induction}
	\left\{ {\begin{array}{*{20}{c}}
			{{v_0} = 0}, \\ 
			{\begin{array}{*{20}{c}}
					{{v_{i + 1}} =  - {l_{i + 1}}{{\left\lfloor {{{\left\lfloor {{z_{i + 1}}} \right\rceil }^{{\beta _i}}} - {{\left\lfloor {{v_i}} \right\rceil }^{{\beta _i}}}} \right\rceil }^{\frac{{{p_{\left( {i + 1} \right)}} + \kappa }}{{{p_{\left( {i} \right)}}}} \times \frac{1}{{{\beta _i}}}}}}, &{i = 0,1,2.} 
			\end{array}} 
	\end{array}} \right.
\end{align}
From \eqref{eq:induction} 
\begin{align}
	v_1 =  - {l_1}{\left\lfloor {{z_1}} \right\rceil ^{\frac{{{p_2}}}{{{p_1}}}}} =  - {l_1}{\left\lfloor {{z_1}} \right\rceil ^{\frac{2}{3}}}.
\end{align}
Then 
\begin{align}
v_2 & =  - {l_2}{\left\lfloor {{{\left\lfloor {{z_2}} \right\rceil }^{{\beta _1}}} - {{\left\lfloor {{v_1}} \right\rceil }^{{\beta _1}}}} \right\rceil ^{\frac{{{p_2} + \kappa }}{{{p_2}}} \times \frac{1}{{{\beta _1}}}}} \\ &=   - {l_2}{\left\lfloor {{{\left\lfloor {{z_2}} \right\rceil }^{\frac{3}{2}}} + l_1^{\frac{3}{2}}{{\left\lfloor {{z_1}} \right\rceil }}} \right\rceil ^{\frac{1}{3}}}.
\end{align}
Finally
\begin{align}
	{v_3} & =  - {l_3}{\left\lfloor {{{\left\lfloor {{z_3}} \right\rceil }^{{\beta _2}}} - {{\left\lfloor {{v_2}} \right\rceil }^{{\beta _2}}}} \right\rceil ^{\frac{{{p_3} +\kappa }}{{{p_3}}} \times \frac{1}{{{\beta _2}}}}} \\ &=- {l_3}{\left\lfloor {{{\left\lfloor {{z_3}} \right\rceil }^4}+ {l_2}^4{\left\lfloor {{{\left\lfloor {{z_2}} \right\rceil }^{\frac{3}{2}}} + l_1^{\frac{3}{2}}{{\left\lfloor {{z_1}} \right\rceil }}} \right\rceil ^{\frac{4}{3}}}} \right\rceil ^0} \\ & =-l_3 sign\Big( \psi_3(z_1,z_2,z_3)\Big).
\end{align}
where $\psi_3(z_1,z_2,z_3)= {{{\left\lfloor {{z_3}} \right\rceil }^4}+ {l_2}^4{\left\lfloor {{{\left\lfloor {{z_2}} \right\rceil }^{\frac{3}{2}}} + l_1^{\frac{3}{2}}{{\left\lfloor {{z_1}} \right\rceil }}} \right\rceil ^{\frac{4}{3}}}}$. Therefore, the controller $u_0$ can be expressed as
\[ \Rightarrow {u_0} =  - {l_3}sign\Big( \psi_3(z_1,z_2,z_3)\Big)\]
Now, the Lyapunov function $V$ is given. From \cite{Hong}, the Lyapunov function is defined in the following form
\begin{align} \label{Lyap-Hong}  
	\left\{ {\begin{array}{*{20}{c}}
			{V}_i =w_i+ V_{i-1}, \quad i=1,2,3 \\ 
			{\begin{array}{*{20}{c}}
					{w}_i =  \int\limits_{{v_{i - 1}}}^{{z_i}} {{{\left\lfloor s \right\rceil }^{{\beta _{i - 1}}}} - {{\left\lfloor {{v_{i - 1}}} \right\rceil }^{{\beta _{i - 1}}}}ds}.
			\end{array}} 
	\end{array}} \right.
\end{align}
 In view of \eqref{Lyap-Hong}  
\begin{equation*}
 {w_1} = \frac{1}{{1 + {p_2}}}{\left| {{z_1}} \right|^{1 + {r_2}}} = \frac{1}{{1 + \frac{2}{3}}}{\left| {{z_1}} \right|^{1 + \frac{2}{3}}},
\end{equation*}
 which leads
\begin{align}
 V_1 = {w_1} = \frac{3}{5}{\left| {{z_1}} \right|^{\frac{5}{3}}}.
\end{align}
Then
\begin{align}
\begin{split} 
 {w_2} & = \int\limits_{{v_1}}^{{z_2}} {{{\lfloor s \rceil }^{{\beta _1}}} - {{\lfloor {{v_1}} \rceil }^{{\beta _1}}}} ds \\& = \frac{1}{{{\beta _1} + 1}}( {{{| {{z_2}} |}^{{\beta _1} + 1}} + {\beta _1}{{| {{v_1}} |}^{{\beta _1} + 1}}} ) - {z_2}{\lfloor {{v_1}} \rceil^{{\beta _1}}}  \\&
= \frac{1}{{\frac{3}{2} + 1}}( {{{| {{z_2}} |}^{\frac{3}{2} + 1}} + \frac{3}{2}{{\lfloor {{v_1}} \rceil}^{\frac{3}{2} + 1}}} ) - {z_2}{\lfloor {{v_1}}\rceil^{\frac{3}{2}}} ,
\end{split}
\end{align}
with ${v_1} =  - {l_1}{\left\lfloor {{z_1}} \right\rceil ^{\frac{2}{3}}}$, it implies that 
\begin{align}
{w_2} = \frac{2}{5}{\left| {{z_2}} \right|^{\frac{5}{2}}} + \frac{3}{5}l_1^{\frac{5}{2}}| {{z_1}} |^{\frac{5}{3}} + l_1^{\frac{3}{2}}{z_2}{\lfloor {{z_1}}\rceil}.
\end{align}
Therefore,
\begin{align}
\begin{split}
{ V_2} &= { V_1} + w_2
\\& = \frac{3}{5}{\left| {{z_1}} \right|^{\frac{5}{3}}}+ \frac{2}{5}{\left| {{z_2}} \right|^{\frac{5}{2}}} + \frac{3}{5}l_1^{\frac{5}{2}}| {{z_1}} |^{\frac{5}{3}} + l_1^{\frac{3}{2}}{z_2}{\lfloor {{z_1}}\rceil}.
	\end{split} 
\end{align}
Finally, 
\begin{align} 
\begin{split} 
{w_3} &= \int\limits_{{v_2}}^{{z_3}} {{{\lfloor s \rceil }^{{\beta _2}}} - {{\lfloor {{v_2}} \rceil }^{{\beta _2}}}} ds \\&
= \frac{1}{{{\beta _2} + 1}}( {{{| {{z_3}} |}^{{\beta _2} + 1}} + {\beta _2}{{| {{v_2}} |}^{{\beta _2} + 1}}} ) - {z_3}{\lfloor {{v_2}} \rceil ^{{\beta _2}}} ,
\end{split} 
 \end{align}
with ${\beta _2} = 4$; ${v_2} =  - {l_2}{\left\lfloor {{{\left\lfloor {{z_2}} \right\rceil }^{\frac{3}{2}}} + l_1^{\frac{3}{2}}{{\left\lfloor {{z_1}} \right\rceil }}} \right\rceil ^{\frac{1}{3}}}$, it leads
\begin{align}
\begin{split} 
{w_3} &= \frac{1}{5}\left( {{{\left| {{z_3}} \right|}^5} + 4 l_2^5{\left\lfloor {{{\left\lfloor {{z_2}} \right\rceil }^{\frac{3}{2}}} + l_1^{\frac{3}{2}}{{\left\lfloor {{z_1}} \right\rceil }}} \right\rceil ^{\frac{5}{3}}}} \right) \\&+ {z_3}l_2^4\left\lfloor { {{{\left\lfloor {{z_2}} \right\rceil }^{\frac{3}{2}}} + l_1^{\frac{3}{2}}{{\left\lfloor {{z_1}} \right\rceil }}} } \right\rceil ^{\frac{4}{3}}	
\\& = \frac{1}{5}{\left| {{z_3}} \right|^5} + \frac{4}{5}l_2^5{\left\lfloor {{{\left\lfloor {{z_2}} \right\rceil }^{\frac{3}{2}}} + l_1^{\frac{3}{2}}{{\left\lfloor {{z_1}} \right\rceil }}} \right\rceil ^{\frac{5}{3}}} \\&+ {z_3}l_2^4\left\lfloor { {{{\left\lfloor {{z_2}} \right\rceil }^{\frac{3}{2}}} + l_1^{\frac{3}{2}}{{\left\lfloor {{z_1}} \right\rceil }}} } \right\rceil ^{\frac{4}{3}}	.
\end{split} 
\end{align}
Thus, the Lyapunov function $V$ is given by
\begin{align}
\begin{split} 
V ={ V_3} & = \frac{3}{5}{\left| {{z_1}} \right|^{\frac{5}{3}}}+ \frac{2}{5}{\left| {{z_2}} \right|^{\frac{5}{2}}} + \frac{3}{5}l_1^{\frac{5}{2}}| {{z_1}} |^{\frac{5}{3}} + l_1^{\frac{3}{2}}{z_2}{\lfloor {{z_1}}\rceil}+ \frac{1}{5}{\left| {{z_3}} \right|^5} + 
\\& \frac{4}{5}l_2^5{\left\lfloor {{{\left\lfloor {{z_2}} \right\rceil }^{\frac{3}{2}}} + l_1^{\frac{3}{2}}{{\left\lfloor {{z_1}} \right\rceil }}} \right\rceil ^{\frac{5}{3}}} + {z_3}l_2^4\left\lfloor { {{{\left\lfloor {{z_2}} \right\rceil }^{\frac{3}{2}}} + l_1^{\frac{3}{2}}{{\left\lfloor {{z_1}} \right\rceil }}} } \right\rceil ^{\frac{4}{3}}.
\end{split} 
\end{align}

\end{appendices}
\end{document}